\documentclass[11pt,notitlepage]{amsart}
\usepackage{amsmath}
\usepackage{amsthm}
\usepackage{amsfonts}
\usepackage{amssymb}
\usepackage{graphicx}
\usepackage{cite}
\usepackage{hyperref}
\usepackage{subfigure}
\usepackage{graphicx, tikz}
\usepackage{epstopdf}
\usepackage{enumitem} 
\usepackage{float}
\usepackage{array}

\newcolumntype{M}[1]{>{\centering\arraybackslash}m{#1}}
\graphicspath{{C:/Users/amill/Pictures/Latex/}}
\newtheorem{theorem}{Theorem}[section]

\newtheorem{lemma}[theorem]{Lemma}

\newtheorem{corollary}[theorem]{Corollary}
\newtheorem{definition}[theorem]{Definition}
\newtheorem{remark}[theorem]{Remark}

\theoremstyle{plain} 
\theoremstyle{definition} \newtheorem*{definition*}{Definition}
\theoremstyle{plain} \newtheorem*{fact*}{Fact}
\theoremstyle{remark} \newtheorem*{remark*}{Remark}
\theoremstyle{definition} \newtheorem*{example*}{Example}
\theoremstyle{definition} 
\theoremstyle{definition} \newtheorem*{notation*}{Notation}

\newcommand{\dx}{\text{ dx}}
\newcommand{\dy}{\text{ dy}}

\newcommand{\RN}{\mathbb{R}}

\newcommand{\ra}{\rightarrow}

\newcommand{\Vh}{V_h}
\newcommand{\vh}{v_h}
\newcommand{\Vhzero}{V_h^0}
\newcommand{\Omegah}{\Omega_h}

\newcommand{\Ghz}{G_h^z}

\newcommand{\regdeltaz}{\tilde{\delta}^{z}}

\newcommand{\reggreenz}{\widetilde{G}^{z}}
\newcommand{\regdeltax}{\tilde{\delta}^{x}}
\newcommand{\regdeltaxk}{\tilde{\delta}^{x_k}}

\newcommand{\dgreenz}{G_h^{z}}
\newcommand{\Om}{\Omega}
\newcommand{\Omh}{\Omega_h}

\newcommand{\Holder}{H\"{o}lder }
\newcommand{\Poincare}{Poincar\'{e}}
\newcommand{\veccv}{\vec{v}}
\newcommand{\veccvT}{\vec{v}^\top}
\newcommand{\disLaplace}{\Delta_h}
\newcommand{\Tau}{\mathcal{T}}

\begin{document}

\title[Positivity of Discrete Green's Function]
{On the Positivity of the Discrete Green's Function for Unstructured Finite Element Discretizations in Three Dimensions.}

\author[Miller]{A. Miller}
\address{Department of Mathematics,
	Bridgewater State University,
	Bridgewater,
	MA~02325, USA.}
\email{andrew.miller@bridgew.edu}

\begin{abstract}
	The aim of this paper is twofold. First, we prove $L^p$ estimates for a regularized Green's function in three dimensions. We then establish new estimates for the discrete Green's function and obtain some positivity results. In particular, we prove that the discrete Green's functions with singularity in the interior of the domain cannot be bounded uniformly with respect of the mesh parameter $h$. Actually, we show that at the singularity the discrete Green's function is of order $h^{-1}$, which is consistent with the behavior of the continuous Green's function. In addition, we also show that the discrete Green's function is positive and decays exponentially away from the singularity. We also provide numerically persistent negative values of the discrete Green's function on Delaunay meshes which then implies a discrete Harnack inequality cannot be established for unstructured finite element discretizations. 
\end{abstract}

\maketitle

%%%%%%%%%%%%%%%%%%%%%%%%%%%%%%%%%%%%%%%%%%%%%%%%%%%%%%%%%%%%%%%%%%%%%%%%%%%%%%%%%%%%%%%%%%%%%%%%%%%%%%%%%%%%%%%%%%%%%%%%%%%%%%%%%%%%%%%%%%%%%%%%%%%%%%%%%%%%%%%%%%%%%%%%%%%%%%%

\section{Introduction}\label{sec. Intro}

Let $\Omega\subset \RN^N$, $N=3$, be a convex bounded domain with sufficiently smooth boundary. We consider the following  Laplace equation;
\begin{align} 
	\begin{split} \label{eq: original problem}
		-\Delta u = 0, & \quad \text{ in } \Omega  \\
		u = b, & \quad \text { on } \partial \Omega.
	\end{split}
\end{align}
Here we assume the boundary data $b\in C(\partial\Om)$ and $b\geq 0$.
%\begin{definition}\label{def: harmonic function}
%	A function satisfying \eqref{eq: original problem} is called a \textit{harmonic function} on $\Omega$. 
%\end{definition}
Harmonic functions in the classical sense have been well studied across mathematics, mathematical physics, and stochastic processes. Many results are established for harmonic functions, including: regularity results, Liouville's Theorem, weak and strong maximum principles, as well as a mean value property for example. An important and well known result for non-negative harmonic functions is the Harnack inequality\cite{MR1510427}. The theorem essentially states that any two values of a non-negative harmonic function are comparable with a constant independent of the particular function itself. The Harnack inequality has been useful in the analysis of fully non-linear elliptic problems. This problem was then extended by Moser in 1961 \cite{MR159138} for uniformly elliptic equations in divergence form with bounded measurable coefficients under the assumption that the eigenvalues of the matrix operator are bounded from above and below. Later in 1980, Krylov and Safonov extended the result to elliptic equations in non-divergence form with bounded measurable coefficients \cite{MR563790}. In addition there is literature on the Harnack inequality outside of the Euclidean setting on $\RN^N$ including, probability \cite{MR1127149}, graph theory \cite{MR3104559}, and infinite dimensional operators \cite{MR2984082}. 

The same cannot be said about the discrete setting, especially when the discretization is less structured. To the author's knowledge there are only two existing papers on the Harnack inequality in the finite element literature. The first is a paper by Aguilera and Caffarelli\cite{MR945780} where they adapted the continuous De Giorgi-Nash-Moser iteration technique. However, this technique required the discrete maximum principle as well as other additional geometric constraints on the discretization. In 2014, Leykekhman and Pruitt established a discrete Harnack inequality \cite{MR3614014} on general quasi-uniform meshes in two dimensions only requiring a mesh condition near the boundary of the domain. Their technique involved using a discrete Green's representation for discrete harmonic functions, upper and lower bounds on both the continuous Green's function and discrete Green's functions, as well as applying analytical tools from Sobolev space theory and finite element theory.  

%We believe a discrete Harnack Inequality can be used to prove discrete \Holder estimates as well as to analyze the discrete non-linear elliptic problem.

Motivated by the Harnack inequality, in this paper we will work towards establishing criteria for the non-negativity of the discrete Green's function including our main result, showing that the discrete Green's function cannot be uniformly bounded at the singularity. We will also prove pointwise upper bounds and establish some $L^p$ estimates for the discrete Green's function. We will adapt techniques used by Leykekhman and Pruitt in two dimensions, however, some two dimensional techniques do not hold in three dimensions. In this case we introduce new analytical techniques while working to keep the discretization of the original space as general as possible. For example, requiring the stiffness matrix to be a $M$-matrix would give positivity of the discrete Green's function, but with significant constraints on the triangulation that many mesh refinement schemes don't follow.

The rest of the paper is organized as follows. In section \ref{sec. prelims} we  introduce the problem as well as any relative definitions and preliminaries from finite elements and Green's functions. Section \ref{reggreenestimates} will concern estimates for the regularized Green's function which have been an important tool for optimal control problems. In section \ref{sec. dgreen results} we present some new estimates for the discrete Green's function including an interesting result that shows the function decays exponentially towards the boundary, analogous to the behavior of the classical Green's function. Section \ref{sec. dgreen positivity} will present our results on the positivity of the discrete Green's function including a new analytical technique used at the ``singularity".  Finally, in \ref{numerical} we will also establish numerically persistent negative values for the discrete Green's function on a polyhedral domain, showing an inconsistency in the finite element approximation of the classical Green's function, suggesting a discrete Harnack inequality cannot be established for unstructured meshes. 

%Only include if publishing results on the regularzed green's function.

%In addition we will provide results of independent interest on a ``smooth" Green's function, described in Chapter \ref{ch greens function}. In three dimensions these results may be used in optimal control problems or results concerning the discrete maximum principle. In chapter \ref{ch numerical}, we will provide numerical results for the discrete Green's function on polyhedral domains.

%%%%%%%%%%%%%%%%%%%%%%%%%%%%%%%%%%%%%%%%%%%%%%%%%%%%%%%%%%%%%%%%%%%%%%%%%%%%%%%%%%%%%%%%%%%%%%%%%%%%%%%%%%%%%%%%%%%%%%%%%%%%%%%%%%%%%%%%%%%%%%%%%%%%%%%%%%%%%%%%%%%%%%%%%%%%%%%%%%%%%%%%%%%%%%%%%%%%%%%%%%%%%%%%%%%%%%%%%%%%%%%%%%%%%%%%%%%%%%%%%%%%%%%%%%%%%%%%%%%%%%%%%%%%%%%%%%%%%%%%%%%%%%%%%%%%%%%%%%%%%%%%%%%%%%%%%%%%%%%%%%%%%%%%%%%%%
%We will also make use of concepts such as shape regularity and quasi-uniformity as well as results such as Galerkin Orthogonality and inverse estimates that can be found in the finite element literature (see, for example \cite{MR2373954} or \cite{MR3015004}). (This was after the first sentance).

\section{Preliminaries}\label{sec. prelims}
Throughout the paper we make use of standard Sobolev and finite element notation. Let $0<h<1$ and $\{\mathcal{T}_h\}$ be quasi-uniform shape regular family of triangulations of size $h$ for the computational domain $\Omh\subset\Om$. We define $V_h(\Omh)$ (resp. $V_h^0(\Omh)$ for zero boundary condition) the set of all continuous functions on $\Omh$ that are linear (affine) when restricted to each tetrahedra in $\Tau$. Let $\{\phi_i\}_{i=1}^{n+m}$ denote the standard nodal basis function for $\Vh(\Omh)$ and let $x_i$ for $i\in\{1,...,n\}$ denote interior nodes and $x_j$ for $j\in\{n+1,...,n+m\}$ denote boundary nodes (those that lie on the boundary).

We then define $u_h\in\Vh$ to be the solution of the problem,
\begin{align}\label{eq: discreteformulation}
	\begin{split}(\nabla u_h,\nabla v_h)_{\Omega_h}=0,&\quad\forall v_h\in\Vhzero(\Omegah),\\
		u_h=I_hb,&\quad \text{on}\;\;\partial\Omega_h.\end{split}
\end{align}
Where $I_hb$ is an interpolant of the form
	\begin{equation*}
		I_hb=\sum_{k=n+1}^{n+m}b_k\phi_k.
	\end{equation*}
\begin{definition}\label{def: discrete harmonic}
	A function $u_h$ satisfying \eqref{eq: discreteformulation} is said to be \textit{discrete harmonic} on $\Omh$.
\end{definition}

We then have the \textit{classical version of the Green's function} for problem \eqref{eq: original problem}. 
\begin{definition}
		Define $G^z(x)$ to be the solution to the following problem;
	\begin{align} 
		\begin{split} \label{eq: greens problem}
			-\Delta G^z(x) = \delta^z(x), & \quad \text{ in } \Omega  \\
			G^z(x) = 0, & \quad \text { on } \partial \Omega.
		\end{split}
	\end{align}
	Here $z\in\Om$ is a fixed point and $\delta^z$ is the delta distribution.
\end{definition}
It is well known that a solution $u$ to Laplace's problem can be represented by  Green's function representation. We will then require upper and lower bounds on the Green's function. The proof of the following result for general second order elliptic equations can be found in \cite{MR0244611}.
\begin{lemma}                          \label{lemma: Green}
	Let $G^z(x)$ denote the Green's function of the Laplace equation on  $\Om\subset \mathbb{R}^N$,  $N\geq 3$. Then the following estimates hold,
	\begin{subequations}         \label{eq: Greens}
		\begin{align}
			|G^z(x)| &\le 
			\begin{array}{ll}
				C |x-z|^{2-N}, & \text{}
			\end{array}
			\label{eq: estimates for Greens}\\
			|\nabla^\alpha_x \nabla^\beta_z G^z(x)| &\le {C|x-z|^{2-N-|\alpha|-|\beta|}}, \quad |\alpha|+|\beta|\geq 1. \label{eq: estimates for derivative of Greens}
		\end{align}
	\end{subequations}
\end{lemma}
There also exists a lower bound for the Green's function which can be found in section 7 of \cite{MR161019}.
\begin{lemma}\label{lemma: lower bound greens}
	Let $G^z(x)$ denote the Green's function of the Laplace equation on  $\Om\subset \mathbb{R}^N$, $N\geq 3$. Then the following estimate holds for some $C>0$,
	\begin{equation}\label{eq: lower bound greens}
		C|x-z|^{2-N}\leq G^z(x).
	\end{equation}
\end{lemma}
We note that when $N=3$ we have the following estimates on $G^z(x)$,
\begin{equation}\label{eq: Green estiamtes}
	C_1|x-z|^{-1}\leq G^z(x)\leq C_2|x-z|^{-1}.
\end{equation}
The \textit{discrete Green's function}, $\Ghz(x)$, is simply the finite element approximation to $G^z(x)$ and can be defined in the following way.
\begin{definition}\label{def: discretegreen}
	The \textit{discrete Green's function} with singularity at $z$ is the function $\Ghz(x)\in\Vhzero(\Omh)$ satisfying
	\begin{equation}\label{eq: discretegreen}
		(\nabla\Ghz,\nabla v_h)_{\Om}=v_h(z),\;\;\;\forall v_h\in\Vhzero(\Omh).
	\end{equation}
\end{definition}
From this definition it is clear to see that $\Ghz(x)$ is also symmetric by substituting $\vh$ with $G_h^x(z)$ as is $G^z(x)$.

The Ritz projection of a function $u\in H^1_0(\Om)$ onto $\Vhzero(\Omh)$ defined by
\begin{equation}\label{def: Ritzprojection}
	(\nabla R_hu,\nabla v_h)_{\Omh} =(\nabla u,\nabla v_h)_{\Omh},\quad\forall v_h\in\Vhzero(\Omh).
\end{equation}
We will also make use of the discrete Laplace operator.
\begin{definition}\label{def: discrete Laplace}
	The discrete Laplace operator, $\disLaplace:\Vhzero(\Omh)\rightarrow\Vhzero(\Omh)$, is defined by
	\begin{equation}\label{eq: discrete Laplace}
		(-\disLaplace v_h,\chi)_{\Om}=(\nabla v_h,\nabla\chi)_{\Om},\quad\forall\chi\in\Vhzero(\Omh).
	\end{equation}
\end{definition}
%The discrete Laplace operator is an analog of the continuous Laplace operator defined on a discrete subspace of functions, sometimes refereed to as the Laplacian matrix.

In the analysis we will also need a \textit{regularized Green's function}. To define this we first need a \textit{regularized delta function}.
\begin{definition}\label{def: regdelta}
	Let $\regdeltaz(x)\geq 0$ denote a regularized delta function supported in an element $\tau_0$ containing $z$ defined by,
	\begin{equation}\label{eq: regdelta}
		(\regdeltaz,v_h)_{\Om}=(\regdeltaz,v_h)_{\tau_0}=v_h(z),\;\;\;\forall v_h\in\Vh(\Omh).
	\end{equation}
\end{definition}
For an explicit construction of $\regdeltaz$ see \cite{MR3734696}.
In addition we also have the following estimates for $\regdeltaz$.	There exists a constant $C$ independent of $z$ such that,
	\begin{equation}\label{eq: regdelta estimates}
		\|\regdeltaz\|_{W^s_p(\tau_0)}\leq Ch^{-s-N(1-\frac{1}{p})},\;\;\;1\leq p\leq\infty,\;\;\;s=0,1.
	\end{equation}

Therefore we get in particular; $\|\regdeltaz\|_{L^1(\Om)}\leq C$, $\|\regdeltaz\|_{L^2(\Om)}\leq Ch^{-\frac{3}{2}}$, and $\|\regdeltaz\|_{L^\infty(\Om)}\leq Ch^{-3}$. 
\begin{remark}\label{rmk: regdelta is 1}
	Due to the construction of $\regdeltaz$ we have $\regdeltaz>0$ which implies $\int_{\tau_0}\regdeltaz\dx=1$. We will use this observation periodically throughout the analysis. See \cite{MR3734696}.
\end{remark}
Additionally,  defining the $L^2$ \textit{projection} of a function $u\in L^2(\Om)$ onto $\Vhzero(\Omh)$ by,
\begin{equation}\label{eq: L2 projection}
	(P_hu,\vh)_\Om=(u,\vh)_\Om,\quad\forall\vh\in\Vhzero(\Omh).
\end{equation}
Then it has been shown that the $L^2$ projection of $\regdeltaz$ has exponential decay away from $z$. The following is Lemma 7.2 in \cite{MR1115238}.
\begin{lemma}\label{lemma: L2 proj regdelta}
	Given a quasi-uniform triangulation $\Omh$ there exists  constants $C$ and $c$ such that
	\begin{equation}\label{eq: L2 proj regdelta}
		|P_h\regdeltaz(x)|\leq Ch^{-N}e^{-\frac{c|x-z|}{h}}.
	\end{equation}
\end{lemma}
Using $\regdeltaz$, we are able to define a regularized Green's function.
\begin{definition}
	The regularized Green's function, $\reggreenz(x)$, is defined by
	\begin{align}\begin{split}\label{eq: reggreen def}
			-\Delta\reggreenz&=\regdeltaz,\quad\text{in }\Om\\
			\reggreenz&=0,\quad\text{on }\partial\Om.
		\end{split}
	\end{align}
\end{definition}
Integration by parts gives us an alternative definition for $\reggreenz$.
\begin{equation}\label{eq: variational reggreen}
	(\nabla\reggreenz,\nabla\vh)_{\Om}=(\regdeltaz,\vh)_{\Om}=v_h(z),\quad\forall\vh\in\Vhzero.
\end{equation}
We note that $\Ghz=R_h\reggreenz=R_hG^z$. This implies that $\Ghz(x)$ is the finite element approximation of both $G^z(x)$ and $\reggreenz(x)$.

\section{Estimates for the Regularized Green's Function}\label{reggreenestimates}
In this section we shall present some results concerning the regularized Green's function. It is important to note that these results are valid for $N=3$ only. Some of the results are used in the analysis of the discrete Green's function while others are of independent interest.

Our first lemma shows that $\reggreenz$ cannot be uniformly bounded in $h$. 
\begin{lemma}\label{lemma: L2 inner reggreen redelta below}
	There exists a constant $C$ independent of $h$ and $z$ such that
	\begin{equation}\label{eq: L2 inner reggreen redelta below}
		(\reggreenz,\regdeltaz)_{\Om}\geq Ch^{-1}.
	\end{equation}
\end{lemma}
\begin{proof}
	Let $\tau_0$ be the element such that $\regdeltaz$ is supported in. Then by Green's representation, for $x\in B_z(ch)$, $c>1$ such that $B_z(ch)\subset\Om$, 
	\begin{align*}
		\reggreenz(z)&=\int_{\tau_0}G(z,x)\regdeltaz(x)\dx\\
		&\geq C\int_{\tau_0}|x-z|^{-1}\regdeltaz(x)\dx\\
		&\geq Ch^{-1}\int_{\tau_0}\regdeltaz(x)\dx=Ch^{-1}.
	\end{align*}
	Therefore we get,
	\begin{equation*}
		(\reggreenz,\regdeltaz)_{\Om}=\int_{\tau_0}\reggreenz\regdeltaz\dx\geq Ch^{-1}\int_{\tau_0}\regdeltaz\dx=Ch^{-1}.
	\end{equation*}
	
\end{proof}
We now show an upper bound, independent of $h$, on the $L^2$ inner product of $\regdeltaz$ and $\reggreenz$.
\begin{lemma}\label{lemma: L2 inner reggreen and regdelta}
	There exists a constant $C$ independent of $h$ and $z$ such that
	\begin{equation}\label{eq: L2 inner reggreen and regdelta}
		(\reggreenz,\regdeltaz)_{\Om}\leq Ch^{-1}.
	\end{equation}
\end{lemma}
\begin{proof}
	Let $\tau_0$ be the element such that $\regdeltaz$ is supported in. Then by Green's representation we have,
	\begin{align*}
		\reggreenz(y)=\int_{\tau_0}G(y,x)\regdeltaz(x)\dx&\leq C_G\int_{\tau_0}|x-y|^{-1}\regdeltaz(x)\dx\\
		&\leq C\|\regdeltaz\|_{L^\infty(\tau_0)}\cdot\int_{\tau_0}|x-y|^{-1}\dx.\\
		&\leq Ch^{-3}\int_{\tau_0}|x-y|^{-1}\dx\text{ by }\eqref{eq: regdelta estimates}.
	\end{align*}
 Now we shall convert to spherical coordinates with $\rho=|x-y|$ and choose $c>1$ such that $\tau_0\subset B_{ch}(z)$ to get,
	\begin{align*}
	 Ch^{-3}\int_{B_{ch}(z)}|x-y|^{-1}\dx&=Ch^{-3}\int_{0}^{ch}\rho^{-1}\rho^2\;d\rho\\
		&\leq Ch^{-1}.
	\end{align*}
	Therefore $\reggreenz(y)\leq Ch^{-1}$ for $|x-y|\leq ch$. Thus,
	\begin{equation*}
		(\reggreenz,\regdeltaz)_{\Om}=\int_{\tau_0}\reggreenz\regdeltaz\dx\leq Ch^{-1}\underbrace{\int_{\tau_0}\regdeltaz\dx}_{=1}=Ch^{-1}.
	\end{equation*}
\end{proof}
As a result of Lemmas \ref{lemma: L2 inner reggreen redelta below} and \ref{lemma: L2 inner reggreen and regdelta} we get the following corollary.
\begin{corollary}\label{lemma: bounds on reggreen}
	There exist constants $C_1$ and $C_2$ independent of $h$ and $z$ with $x\in B_{ch}(z)$, $c>1$ such that $B_z(ch)\subset\Om$, we have
	\begin{equation}\label{eq: bounds on reggreen}
		C_1h^{-1}\leq \reggreenz(x)\leq C_2h^{-1},
	\end{equation}
	and in particular we see,
	\begin{equation}\label{eq: bounds on reggreen node}
		C_1h^{-1}\leq \reggreenz(z)\leq C_2h^{-1},
	\end{equation}
\end{corollary}
The next lemma concerns an upper bound for $\nabla\reggreenz$ in the $L^\infty$ norm.
\begin{lemma}\label{lemma: Linfty grad reggreen}
	There exists a constant $C$ independent of $z$ and $h$ such that
	\begin{equation}\label{eq: Linfty grad reggreen}
		\|\nabla\reggreenz\|_{L^\infty(\Om)}\leq Ch^{-2}.
	\end{equation}
\end{lemma}
\begin{proof}
	Let $y\in\Omh$. We shall observe two cases.\\
	
	\textbf{Case 1:} Let $c>1$ and $y\in\Omh\backslash B_{ch}(z)$.\\
	We appeal to a Green's function representation;
	\begin{align*}
		|\nabla\reggreenz(y)|&=\bigg|\int_{\tau_0}\nabla_y G(y,x)\regdeltaz(x)\dx\bigg|\\
		&\leq\int_{\tau_0}|\nabla_y G(y,x)||\regdeltaz(x)|\dx\\
		&\leq C\int_{\tau_0}|x-y|^{-2}|\regdeltaz(x)|\dx\\
		&\leq Ch^{-2}\|\regdeltaz\|_{L^1(\Omh)},\;\;\;\text{since }|x-y|>ch,\\
		&\leq Ch^{-2}.
	\end{align*}
	Where the last line follows from estimate \eqref{eq: regdelta estimates}. We shall now proceed to case 2.\\
	
	\textbf{Case 2:} Let $c>1$ and $y\in B_{ch}(z)$.\\
	We again appeal to Green's representation as in case 1 in which we still have,
	\begin{equation*}
		|\nabla\reggreenz(y)|\leq C\int_{\tau_0}|x-y|^{-2}|\regdeltaz(x)|\dx.
	\end{equation*}
	Now we will apply the $L^\infty$ estimate \eqref{eq: regdelta estimates} and the \Holder's inequality to get,
	\begin{equation*}
		\int_{\tau_0}|x-y|^{-2}|\regdeltaz(x)|\dx\leq Ch^{-3}\int_{\tau_0}|x-y|^{-2}\dx.
	\end{equation*}
	We shall now apply spherical coordinates with $|x-y|=\rho$ to get,
	\begin{align*}
		Ch^{-3}\int_{\tau_0}|x-y|^{-2}\dx&\leq Ch^{-3}\int_{B_{ch}(z)}|x-y|^{-2}\dx\\
		&=Ch^{-3}\int_{0}^{ch}\rho^{-2}\rho^2\;d\rho\\
		&\leq Ch^{-3}h=Ch^{-2}.
	\end{align*}
	Since $y$ is arbitrary, the result follows.
	
\end{proof}

%In addition we will use Lemma \ref{lemma: L2 inner reggreen and regdelta} to show an upper bound $\nabla\reggreenz$ in the $L^2$ norm.
%\begin{lemma}\label{lemma: L2 grad reggreen}
%	There exists a constant $C$ independent of $h$ and $z$ such that
%	\begin{equation}\label{eq: L2 grad reggreen}
%		\|\nabla\reggreenz\|_{L^2(\Omh)}\leq Ch^{-\frac{1}{2}}.
%	\end{equation}
%\end{lemma}
%\begin{proof}
%	Observe,
%	\begin{align*}
%		\|\nabla\reggreenz\|^2_{L^2(\Omh)}&=\int_{\Omh}\nabla\reggreenz\cdot\nabla\reggreenz\dx\\
%		&=\int_{\Omh}\reggreenz\cdot\regdeltaz\dx,\;\;\;\text{by }\eqref{eq: variational reggreen},\\
%		&\leq Ch^{-1},\;\;\;\text{by }\ref{eq: L2 inner reggreen and regdelta}.
%	\end{align*}
%	Therefore the result follows.
%	
%\end{proof} 

Our next result concerns the $L^1$ norm of the regularized Green's function which requires a different proof technique. The result may be of independent interest, potentially in showing best approximation results.
\begin{lemma}\label{lemma L1 grad reggreen}
	There exists a constant $C$ independent of $h$ and $z$ such that 
	\begin{equation}\label{eq L1 grad reggreen}
		\|\nabla\reggreenz\|_{L^1(\Omh)}\leq C.
	\end{equation}
\end{lemma}
\begin{proof}
	We shall proceed using a dyadic decomposition of $\Omh$. With out loss of generality, assume $diam(\Omh)\leq 1$. Then let $d_j=2^{-j}$ and
	\begin{equation*}
		\Omh=\Om_*\cup\bigcup\limits_{j=0}^J\Om_j,
	\end{equation*}
	where
	\begin{align*}
		\Om_*&=\{x\in\Omh\mid |x-z|\leq Kh\},\\
		\Om_j&=\{x\in\Omh\mid d_{j+1}\leq|x-z|\leq d_j\},
	\end{align*}
	with $K$ possibly chosen later and $J$ such that $2^{-J}\leq Kh\leq2^{-J+1}$. This then gives
	\begin{equation*}
		\|\nabla\reggreenz\|_{L^1(\Omh)}\leq\underbrace{\sum_{j=0}^{J}\|\nabla\reggreenz\|_{L^1(\Om_j)}}_{(1)}+\underbrace{\|\nabla\reggreenz\|_{L^1(\Om_*)}}_{(2)}.
	\end{equation*}
	We will now estimate $(1)$ and $(2)$ individually. For $(1)$ applying the \Holder's inequality first,
	\begin{align*}
		\|\nabla\reggreenz\|_{L^1(\Om_j)}&\leq d_j^3\|\nabla\reggreenz\|_{L^\infty(\Om_j)}\\
		&\leq d_j^3\int_{\tau_0}C|x-z|^{-2}|\regdeltaz|\dx\\
		&\leq Cd_j^3d_{j+1}^{-2}\int_{\tau_0}|\regdeltaz|\dx\\
		&\leq Cd_j^3d_j^{-2}\leq Cd_j.
	\end{align*}
	Where the second line follows from Lemma \ref{lemma: Linfty grad reggreen}. Therefore, 
	\begin{equation*}
		\sum_{j=0}^{J}\|\nabla\reggreenz\|_{L^1(\Om_j)}\leq \sum_{j=0}^{J}Cd_j\leq C.
	\end{equation*} To estimate $(2)$ we again appeal to the \Holder's inequality and Lemma \ref{lemma: Linfty grad reggreen} to get, 
	\begin{align*}
		\|\nabla\reggreenz\|_{L^1(\Om_*)}&\leq K^3h^3\|\nabla\reggreenz\|_{L^\infty(\Om_*)}\\
		&\leq CK^3h^3h^{-2}\\
		&\leq Ch.
	\end{align*}
	In fact, as $h\ra0$, we see that $\|\nabla\reggreenz\|_{L^1(\Om_*)}\ra0$.
	
\end{proof}

%\begin{theorem}\label{thm: interpolation inequality}
%	Assume $1\leq s\leq p \leq t\leq\infty$ and 
%	\begin{center}
%		$\frac{1}{p}=\frac{\theta}{s}+\frac{1-\theta}{t}$.
%	\end{center}
%	Suppose also $u\in L^s(\Om)\cap L^t(\Om)$, then $u\in L^p(\Om)$ and 
%	\begin{equation*}
%		\|u\|_{L^p(\Om)}\leq\|u\|^\theta_{L^s(\Om)}\|u\|^{1-\theta}_{L^t(\Om)}.
%	\end{equation*}
%\end{theorem}

%Now with $s=1$ and $t=\infty$ we get $\theta=\frac{1}{p}$ and
%\begin{align*}
%	\|\nabla\reggreenz\|_{L^p(\Omh)}&\leq\|\nabla\reggreenz\|^{\frac{1}{p}}_{L^1(\Omh)}\|\nabla\reggreenz\|^{1-\frac{1}{p}}_{L^\infty(\Omh)}\\
%	&\leq Ch^{-2(1-\frac{1}{p})}\\
%	&=Ch^{\frac{2}{p}-2}.
%\end{align*}
%%Thus we get the following corollary.
%\begin{corollary}\label{corollary: Lp grad reggreen}
%	There exists a constant $C$ independent of $h$ and $z$ such that 
%	\begin{equation}\label{eq: Lp grad reggreen}
%		\|\nabla\reggreenz\|_{L^p(\Omh)}\leq Ch^{\frac{2}{p}-2},\quad \text{for}\;\; 1\leq p\leq\infty.
%	\end{equation}
%\end{corollary}

%%%%%%%%%%%%%%%%%%%%%%%%%%%%%%%%%%%%%%%%%%%%%%%%%%%%%%%%%%%%%%%%%%%%%%%%%%%%%%%%%%%%%%%%%%%%%%%%%%%%%%%%%%%%%%%%%%%%%%%%%%%%%%%%%%%%%%%%%%%%%%%%%%%%%%%%%%%%%%%%%%%%%%%%%%%%%%%%%%%%%%%%%%%%%%%%%%%%%%%%%%%%%%%%%%%%%%%%%%%%%%%%%%%%%%%%%%%%%%%%%%%%%%%%%%%%%%%%%%%%%%%%%%%%%%%%%%%%%%%%%%%%%%%%%%%%%%%%%%%%%%%%%%%%%%%%%%%%%%%%%%%%%%%%%%%%%

\section{Estimates for the Discrete Green's Function}\label{sec. dgreen results}
In the analysis we will require the following inverse estimate from Theorem (4.5.11) in \cite{MR2373954}, for $\vh\in\Vh$, there exists $C>0$ such that
\begin{equation}\label{eq: inverse L2}
	\|\nabla\vh\|_{L^2(\Om)}\leq Ch^{-1}\|\vh\|_{L^2(\Om)}.
\end{equation}

Using the regularized Green's function we shall show an estimate for $\Ghz(z)$ and then give a resulting corollary. 
\begin{lemma}\label{lemma: dGreen upperbound} 
	There exists a constant $C$ independent of $h$ and $z$ such that
	\begin{equation}\label{eq: dGreen upperbound}
	\Ghz(z)\leq Ch^{-1}.
	\end{equation}
\end{lemma}
\begin{proof}
 Indeed, 
	\begin{align*}
		\Ghz(z)&=(\Ghz,\regdeltaz)_{\Omh}\\
		&=(\nabla\Ghz,\nabla\Ghz)_{\Omh}\\
		&=\|\nabla\reggreenz\|_{L^2(\Omh)}^2-\|\nabla\reggreenz\|_{L^2(\Omh)}^2+\|\nabla\Ghz\|_{L^2(\Omh)}^2\\
		&=(\reggreenz,\regdeltaz)_{\Omh}+\int_{\Omh}(|\nabla\Ghz|^2-|\nabla\reggreenz|^2)\dx\\
		&=(\reggreenz,\regdeltaz)_{\Omh}-\int_{\Omh}(|\nabla\reggreenz|^2-|\nabla\Ghz|^2)\dx\\
		&=(\reggreenz,\regdeltaz)_{\Omh}-(\nabla(\reggreenz-\Ghz),\nabla(\reggreenz+\Ghz))_{\Omh}\\
		&=(\reggreenz,\regdeltaz)_{\Omh}-(\nabla(\reggreenz-\Ghz),\nabla\reggreenz)_{\Omh}\\
		&=(\reggreenz,\regdeltaz)_{\Omh}-(\nabla(\reggreenz-\Ghz),\nabla(\reggreenz-\Ghz))_{\Omh}\\
		&=(\reggreenz,\regdeltaz)_{\Omh}-\|\nabla(\reggreenz-\Ghz)\|^2_{L^2(\Omh)}\\
		&\leq(\reggreenz,\regdeltaz)_{\Omh}=(\reggreenz,\regdeltaz)_{\tau_0}\\
		&\leq Ch^{-1}.
	\end{align*}
	We have used the Galerkin orthogonality and Lemma \ref{lemma: L2 inner reggreen and regdelta}.
\end{proof}

We then have a resulting corollary.
\begin{corollary}\label{corollary: L2 dgreen bound reggreen}
	There exists a constant $C$ such that,
	\begin{equation}\label{eq: L2 dgreen bound reggreen}
		\|\nabla\Ghz\|_{L^2(\Omh)}\leq\|\nabla\reggreenz\|_{L^2(\Omh)}\leq Ch^{-\frac{1}{2}}.
	\end{equation}
\end{corollary}
\begin{proof}
	This follows from showing, by Galerkin orthogonality,
	\begin{equation*}
		\int_{\Omh}|\nabla\Ghz|^2-|\nabla\reggreenz|^2\dx=-\|\nabla(\reggreenz-\Ghz)\|^2_{L^2(\Omh)}\leq 0.
	\end{equation*}
To show that, 
	\begin{equation}\label{eq: L2 grad reggreen}
	\|\nabla\reggreenz\|_{L^2(\Omh)}\leq Ch^{-\frac{1}{2}}.
\end{equation}
Observe,
\begin{align*}
	\|\nabla\reggreenz\|^2_{L^2(\Omh)}&=\int_{\Omh}\nabla\reggreenz\cdot\nabla\reggreenz\dx\\
	&=\int_{\Omh}\reggreenz\cdot\regdeltaz\dx,\;\;\;\text{by }\eqref{eq: variational reggreen},\\
	&\leq Ch^{-1},\;\;\;\text{by }\eqref{eq: L2 inner reggreen and regdelta}.
\end{align*}
Therefore the result follows.
\end{proof}

Next we shall show the discrete Green's function decays exponentially for values away from the singularity. 

\begin{lemma}\label{lemma: dgreen exponential decay}
	For $x\in\Om$ there exists constants $C$ and $c$ independent of $z$ and $h$ such that
	\begin{equation}\label{eq: dgreen exponential decay}
		\dgreenz(x)\leq Ch^{-5}e^{-\frac{c|x-z|}{h}}.
	\end{equation}
\end{lemma}
\begin{proof}
	We shall use the fact that the discrete Laplace operator is bounded as well as estimate \eqref{eq: L2 proj regdelta}. First note that by \eqref{eq: discretegreen}, \eqref{eq: regdelta}, \eqref{eq: L2 projection}, and \eqref{eq: discrete Laplace} we have,
	\begin{equation}\label{eq: dLaplace dGreen ch3}
		-\disLaplace\Ghz=P_h\regdeltaz.
	\end{equation}
	Therefore let $x\in\Om$ be fixed and arbitrary with associated regularized delta function $\regdeltax$, then
	\begin{align*}
		\dgreenz(x)=(\regdeltax,\dgreenz)_{\tau_x}&=(-\Delta_h\regdeltax,P_h\regdeltaz)_{\tau_x}\\
		&\leq C\|\regdeltax\|_{L^1(\tau_x)}\|\Delta_hP_h\regdeltaz\|_{L^\infty(\tau_x)}\\
		&\leq Ch^{-2}\|P_h\regdeltaz\|_{L^\infty(\tau_x)}\\
		&\leq Ch^{-5}e^{-\frac{|x-z|}{h}}.
	\end{align*}
\end{proof}

We first note that this result does not give us insight to the positivity of $\dgreenz(x)$. However, in contrast to Lemma \ref{lemma: dGreen upperbound} we gain a significantly sharper bound on the discrete Green's function for $x$ outside of a ball around the singularity.

%%Need to decide on how to include L2 norm of grad(regG)$%%
Next, we present an upper bound for the discrete Green's function in the $W^1_\infty(\Om)$ norm. 
\begin{lemma}\label{lemma: graddgreen up}
	There exists a constant $C$ independent of $h$ such that
	\begin{equation}\label{eq: graddgreen up}
		\|\dgreenz\|_{W^1_\infty(\Om)}\leq Ch^{-2}.
	\end{equation}
\end{lemma}
\begin{proof}
	We first note that by Theorem 4.5.11 in \cite{MR2373954} we have the following,
	\begin{equation*}
		\|\dgreenz\|_{W^1_\infty(\Om)}\leq Ch^{-\frac{3}{2}}\|\dgreenz\|_{H^1(\Om)}.
	\end{equation*}
	Then since $\dgreenz\in H^1_0(\Om)$ we have,
	\begin{equation*}
		\|\dgreenz\|_{H^1(\Om)}\leq C|\dgreenz|_{H^1(\Om)}.
	\end{equation*}
	Therefore applying Corollary \ref{corollary: L2 dgreen bound reggreen} we see,
	\begin{align*}
		Ch^{-\frac{3}{2}}|\dgreenz|_{H^1(\Om)}&\leq Ch^{-\frac{3}{2}}\|\nabla\reggreenz\|_{L^2(\Omh)}\\
		&\leq Ch^{-2},
	\end{align*}
	and we arrive at the desired result.
	
\end{proof}
This result is most notable close to the singularity as it may suggest, due to Lemma \ref{lemma: dGreen upperbound}, that the discrete Green's function will have a persistent negative value. Otherwise, this result along with Lemma \ref{lemma: dGreen upperbound} can be thought of as an analog to the upper bounds of the continuous elliptic Green's function since this result implies,
\begin{equation*}
	\|\nabla\dgreenz\|_{L^\infty(\Om)}\leq Ch^{-2}.
\end{equation*}

Additionally, since Corollary \ref{corollary: L2 dgreen bound reggreen} and Lemma \ref{lemma: graddgreen up} give,
\begin{align*}
	\|\nabla\dgreenz\|_{L^2(\Om)}&\leq Ch^{-\frac{1}{2}}\quad\text{and}\\
	\|\nabla\dgreenz\|_{L^\infty(\Om)}&\leq Ch^{-2}.
\end{align*}
We can apply the interpolation inequality found in \cite{MR2597943} to get the following Corollary.
\begin{corollary}\label{corollary: Lp grad dgreen}
	There exists a constant $C$ independent of $h$ and $z$ such that 
	\begin{equation}\label{eq: Lp grad dgreen}
		\|\nabla\dgreenz\|_{L^p(\Omh)}\leq Ch^{\frac{3}{p}-2},\quad \text{for}\;\; 2\leq p\leq\infty.
	\end{equation}
\end{corollary}

%%%%%%%%%%%%%%%%%%%%%%%%%%%%%%%%%%%%%%%%%%%%%%%%%%%%%%%%%%%%%%%%%%%%%%%%%%%%%%%%%%%%%%%%%%%%%%%%%%%%%%%%%%%%%%%%%%%%%%%%%%%%%%%%%%%%%%%%%%%%%%%%%%%%%%%%%%%%%%%%%%%%%%%%%%%%%%%%%%%%%%%%%%%%%%%%%%%%%%%%%%%%%%%%%%%%%%%%%%%%%%%%%%%%%%%%%%%%%%%%%%%%%%%%%%%%%%%%%%%%%%%%%%%%%%%%%%%%%%%%%%%%%%%%%%%%%%%%%%%%%%%%%%%%%%%%%%%%%%%%%%%%%%%%%%%%%

\section{On the Positivity of the Discrete Green's Function in 3D}\label{sec. dgreen positivity}
It is well known that the classical Green's function for \eqref{eq: Greens} is positive. However, the discrete Green's function arising from the finite element method is not as straightforward. For $N=2$, Dr\v{a}g\v{a}nescu, et al. \cite{MR2085400}, provided a counterexample showing on general meshes the discrete Green's function may have a persistent negative value. For these meshes, the negative value and the singularity were both an $O(h)$ away from the boundary of the mesh. Later for $N=2$, Leykekhman and Pruitt \cite{MR3614014}, showed on a quasi-uniform, shape regular mesh of a polygonal domain, that if the singularity is of $O(1)$ away from the boundary the discrete Green's function must eventually be non-negative. With the help of a solution representation using the discrete Green's function they were then able to show the non-negativity of discrete harmonic functions which then led to the establishment of a discrete Harnack type inequality. In this section we will work towards showing the positivity of the discrete Green's function in three dimensions. 

Positivity at the singularity is a simple consequence of Definition \ref{def: discretegreen}. By choosing $\vh=\Ghz$ we see that,
\begin{equation*}
	\Ghz(z)=(\nabla\Ghz,\nabla\Ghz)_\Om=|\Ghz|^2_{H^1(\Om)}>0.
\end{equation*}
Considering the estimate \eqref{eq: Green estiamtes} for the classical Green's function one would expect that $\Ghz(z)\ra\infty$ as $h\ra0$. In fact, in Section 5 of \cite{MR2085400}, it was conjectured that,
\begin{equation*}
	\Ghz(z)=O(h^{2-N})\quad\forall z\;\text{(up to the boundary)}.
\end{equation*}
We shall show that this holds for $N=3$, however, the proof can be easily adapted for $N>3$. 
\begin{theorem}\label{lemma: dgreen lowerbound}
	There exists a constant C independent of $h$ and $z$ such that
	\begin{equation}\label{eq: dgreen lowerbound}
		\Ghz(z)\geq Ch^{-1}.
	\end{equation}
\end{theorem}
\begin{proof}
	We again note that by \eqref{eq: discretegreen}, \eqref{eq: regdelta}, \eqref{eq: L2 projection}, and \eqref{eq: discrete Laplace} we have,
	\begin{equation}\label{eq: dLaplace dGreen}
		-\disLaplace\Ghz=P_h\regdeltaz,
	\end{equation}
	where $	-\disLaplace$ is the discrete Laplace operator. Since for all $\vh\in\Vhzero(\Omh)$ we have,
	\begin{equation*}
		(-\disLaplace\Ghz,\vh)_\Om=\underbrace{(\nabla\Ghz,\nabla\vh)_\Om=(\regdeltaz,\vh)_\Om}_{=\vh(z)}=(P_h\regdeltaz,\vh)_\Om.
	\end{equation*}
	Now, for all $\vh\in\Vhzero$ we have by the \Poincare's inequality,
	\begin{equation}\label{eq: dLaplace poincare}
		C_1(\vh,\vh)_\Om\leq(\nabla\vh,\nabla\vh)_\Om=(-\disLaplace\vh,\vh)_\Om.
	\end{equation}
	In addition, the inverse estimate \eqref{eq: inverse L2} gives,
	\begin{equation}\label{eq: dLaplace inverse}
		(-\disLaplace\vh,\vh)_\Om=(\nabla\vh,\nabla\vh)_\Om\leq C_2h^{-2}(\vh,\vh)_\Om.
	\end{equation}
	Combining \eqref{eq: dLaplace poincare} and \eqref{eq: dLaplace inverse} we see that,
	\begin{equation*}
		C_1\leq\dfrac{(\vh,-\disLaplace\vh)_\Om}{(\vh,\vh)_\Om}\leq C_2h^{-2}.
	\end{equation*}
	This then implies the following estimate which holds by Lemma \ref{lemma: inverse matrix bounds}, given in the appendix with proof,
	\begin{equation}\label{eq: inverse dLaplace bounds}
		C_2^{-1}h^2\leq\dfrac{(\vh,(-\disLaplace)^{-1}\vh)_\Om}{(\vh,\vh)_\Om}\leq C_1^{-1}.
	\end{equation}
	Now by \eqref{eq: dLaplace dGreen} and \eqref{eq: inverse dLaplace bounds} we have,
	\begin{align*}
		\Ghz(z)&=(\regdeltaz,\Ghz)_\Om\\
		&=(P_h\regdeltaz,\Ghz)_\Om\\
		&=(P_h\regdeltaz,(-\disLaplace)^{-1}P_h\regdeltaz)_\Om\\
		&\geq Ch^2\|P_h\regdeltaz\|^2_{L^2(\Om)}.
	\end{align*}
	Here the key step is to apply the $L^2$-projection to $\regdeltaz$ in order to project our regularized delta function onto $\Vhzero(\Om)$.
	
	Therefore it remains to show $\|P_h\regdeltaz\|^2_{L^2(\Om)}\geq Ch^{-3}$. Indeed, by \eqref{eq: L2 projection},
	\begin{equation*}
		\|P_h\regdeltaz\|^2_{L^2(\Om)}=(P_h\regdeltaz,P_h\regdeltaz)_\Om=(P_h\regdeltaz,\regdeltaz)_\Om=P_h\regdeltaz(z),
	\end{equation*}
	where $P_h\regdeltaz(x)=\sum_{i=1}^n\gamma_i\phi_i(x)$ by definition of $P_h\regdeltaz\in\Vhzero(\Omh)$ using the nodal basis.
	Now for $z=x_k$, $x_k$ an interior node, we get,
	$P_h\regdeltaxk(x_k)=\sum_{i=1}^n\gamma_i\phi_i(x_k)=\gamma_k$. In addition, for $j=1,...,n$ we have,
	\begin{equation*}
		(P_h\regdeltaxk,\phi_j)_\Om=(\regdeltaxk,\phi_j)_\Om=\phi_j(x_k)=\delta_{kj},
	\end{equation*}
	and therefore we get,
	\begin{equation*}
		(P_h\regdeltaxk,\phi_j)_\Om=(\sum_{i=1}^n\gamma_i\phi_i(x),\phi_j)_\Om=\sum_{i=1}^n\gamma_i(\phi_i,\phi_j)_\Om=\delta_{kj},\quad j=1,...,n.
	\end{equation*}
	Interpreting the system as a matrix equation,
	\begin{equation*}
		M\vec{\gamma}=\vec{e_k},
	\end{equation*}
	where $M$ is the \textit{Mass matrix} with entries $M_{ij}=(\phi_i,\phi_j)_\Om$, $\vec{\gamma}^\top=(\gamma_1,...,\gamma_n)^\top$, and $\vec{e_k}^\top=(0,...,0,1,0,...,0)^\top$ with $1$ in the $k^{th}$ position. Which then implies,
	\begin{equation}\label{eq: gammak equality}
		\gamma_k=\vec{e_k}^\top\vec{\gamma}=\vec{e_k}^\top M^{-1}\vec{e_k}.
	\end{equation}
	Now by properties of the Mass matrix for finite elements in three dimensions, see for example, Rannacher \cite{Rannacher2008} (page 121), there exists constants $c_1$ and $c_2$ independent of $h$ such that for any $\vec{v}\in\RN^n\backslash\{\vec{0}\}$,
	\begin{equation*}
		c_1h^3\leq\dfrac{\vec{v}^\top M\vec{v}}{\vec{v}^\top\vec{v}}\leq c_2h^3.
	\end{equation*}
	Therefore by Lemma \ref{lemma: inverse matrix bounds} we see,
	\begin{equation}\label{eq: Mass Matrix bounds}
		c_2^{-1}h^{-3}\leq\dfrac{\vec{v}^\top M^{-1}\vec{v}}{\vec{v}^\top\vec{v}}\leq c_1^{-1}h^{-3}.
	\end{equation}
	Thus, choosing $\vec{v}=\vec{e_k}$ and substituting \eqref{eq: gammak equality} into \eqref{eq: Mass Matrix bounds} we have shown $\|P_h\regdeltaz\|^2_{L^2(\Om)}\geq Ch^{-3}$ and the result follows.
	
\end{proof}

We shall now show that the discrete Green's function is positive away from the singularity. The proof technique used will give us insight into the scaling of $\Ghz$.
\begin{theorem}\label{lemma: Ghz positive away}
	Assume for some constant $K$ that $|x-z|\geq Kh^{\frac{1}{2}}$. Then $\Ghz(x)>0$.
\end{theorem}
\begin{proof}
	We know using \eqref{eq: regdelta} that,
	\begin{align*}
		\Ghz(x)&=(\Ghz,\regdeltax)_{\tau_x}\\
		&=\underbrace{(\Ghz-G^z,\regdeltax)_{\tau_x}}_{I_1}+\underbrace{(G^z,\regdeltax)_{\tau_x}}_{I_2},
	\end{align*}
	where $\tau_x$ is the element containing the support of $\regdeltax$. Clearly $I_2$ is positive due to the properties of $G^z$ and $\regdeltax$, however, we can find a lower bound for $I_2$. Observe,
	\begin{align*}
		(G^z,\regdeltax)_{\tau_x}&=\int_{\tau_x}G^z(y)\cdot\regdeltax(y)\dy\\
		&\geq C\int_{\tau_x}|x-y|^{-1}\cdot\regdeltax(y)\dy,\quad\text{by }\eqref{eq: lower bound greens}\\
		&\geq C(|z-x|+h)^{-1}\int_{\tau_x}\regdeltax(y)\dy,\quad\text{since }|x-y|\leq ch\\
		&\geq C|z-x|^{-1}>0.
	\end{align*}
	Now we shall show that $I_1\ra0$ as $h\ra0$, therefore showing $\Ghz(x)\approx C|z-x|^{-1}$ when $|x-z|\geq Kh^{\frac{1}{2}}$, showing that the behavior of $\Ghz(x)$ is consistent with the classical Green's function. Indeed, we will use Theorem 6.1 where $K$ is the constant $C_7$ in estimate (6.3) in \cite{MR431753} which states: if $|x-z|\geq Kh$, then
	\begin{equation}\label{eq: 6.3 estimate S and W}
		|\Ghz(x)-G^z(x)|\leq \dfrac{Ch^2\ln\big(\frac{|x-z|}{h}\big)}{|x-z|^3}.
	\end{equation}
	Adopting the notation $\omega=\frac{|x-z|}{h}$ and using the assumption \eqref{eq: 6.3 estimate S and W} becomes,
	\begin{align}\label{eq: 6.3 estimate new}
		\begin{split}
			|\Ghz(x)-G^z(x)|&\leq\dfrac{C\ln(\omega)\omega^{-2}}{|x-z|}\\
			&\leq\dfrac{C\ln(\omega)\omega^{-2}}{Kh^{\frac{1}{2}}}.
		\end{split}
	\end{align}
	Additionally $\omega\geq Kh^{-\frac{1}{2}}$ and observe that $\ln(\omega)\omega^{-2}$ is decreasing for $\omega\geq\sqrt{e}$. Therefore it's maximum must occur at the left endpoint of $\omega\geq Kh^{-\frac{1}{2}}$ for $h$ small enough. Now \eqref{eq: 6.3 estimate new} becomes,
	\begin{align*}
		|\Ghz(x)-G^z(x)|&\leq\dfrac{C\ln(\omega)\omega^{-2}}{Kh^{\frac{1}{2}}}\\
		&\leq \dfrac{C\ln(Kh^{-\frac{1}{2}})K^{-2}h}{Kh^{\frac{1}{2}}}\\
		&\leq C\ln(Kh^{-\frac{1}{2}})h^{\frac{1}{2}},
	\end{align*}
	where $C\ln(Kh^{-\frac{1}{2}})h^{\frac{1}{2}}\ra0$ as $h\ra0$. 
	Thus, $\Ghz(x)>0$.
	
\end{proof}

\section{Numerical Results}\label{numerical}
In this section we shall provide numerical results concerning the positivity of $\dgreenz(x)$. The results were obtained using the MATLAB package iFEM \cite{LC2009} developed by L. Chen.

The first mesh we consider is a cube positioned at the nodes \{(-1,-1,-1), (-1,-1,1), (-1,1,1), (1,-1,-1), (1,1,-1), (-1,1,-1), (1,-1,1), (1,1,1)\} with the singularity at $z=$(-0.75,-0.75,-0.75). The initial triangulation is Delaunay.
\begin{figure}[H]
	\centering
	\includegraphics[scale=.45]{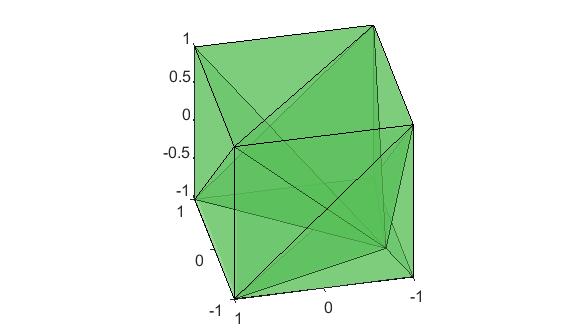}
	\caption{First triangulation of the cubic domain.}
\end{figure}

We apply a uniform refinement algorithm which subdivides each tetrahedron into 8 smaller and similar ``sub-tetrahedrons" of equal volume. To retain shape regularity of the mesh we calculate the lengths of the three interior diagonals of the octahedron and subdivide through the shortest diagonal. 
%Comment out for submission
\begin{figure}[H]
	\centering
	\includegraphics[scale=.3]{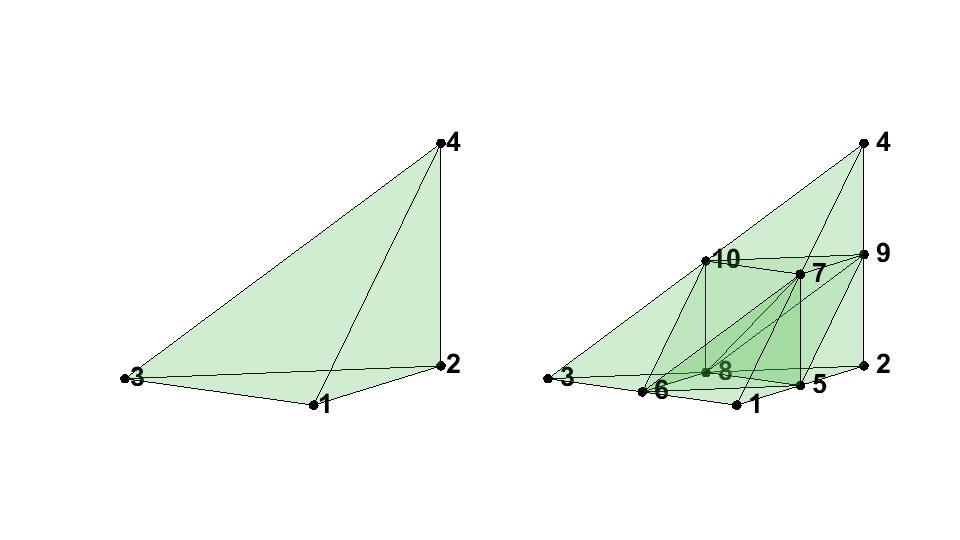}
	\caption{Uniform refinement of a single tetrahedron across shortest diagonal.}
	\label{fig: exuniformrefinel}
\end{figure}
The result of this mesh refinement scheme are shown below. We observe a consistent negative value for the discrete Green's function. Additionally, after the third refinement, we see the number of nodes increases by an approximate factor of 8 and we should have $\Ghz$ increase by an approximate factor of 2. 
\begin{table}[H]
	\begin{center}
		\begin{tabular}{|c|c|c|c|} 
			\hline
			No. Nodes & No. Elements & $\min_{x \in \Omega_h}\Ghz(x)$ & $\max_{x \in \Omega_h}\Ghz(x)$\\
			\hline 
			35 & 96 & -0.00279777 & 0.148468\\ 
			\hline 
			189 & 768 & -0.0252328 & 0.366708\\ 
			\hline 
			1241 & 6144 & -0.0144661 & 0.841241\\ 
			\hline 
			9009 & 49152 & -0.00116218 & 1.84198\\ 
			\hline 
			68705 & 393216 & -0.00249823 & 3.88512\\ 
			\hline 
			536769 & 3.14573e+06 & -0.000698242 & 7.99231\\  
			\hline 
		\end{tabular}
	\end{center}
	\caption{Discrete Green's function values using uniform refinement across shortest diagonal.}
	\label{table: mesh1 uniformrefinel}
\end{table}

We are now motivated to investigate local refinement around the singularity. For the second mesh, we shall again consider a cube positioned at the nodes \{(-1,-1,-1), (-1,-1,1), (-1,1,1), (1,-1,-1), (1,1,-1), (-1,1,-1), (1,-1,1), (1,1,1)\} with the singularity at $z=$(-0.75,-0.75,-0.75). Except we would like to refine locally around the singularity. In this case, we also add in the nodes \{(-.74,-.74,-.74), (-.74,-.74,-.76), (-.74,-.76,-.76), (-.76,-.74,-.76), (-.76,-.76,-.74), (-.74,-.76,-.74), (-.76,-.74,-.74), (-.76,-.76,-.76)\} to generate the initial mesh. We again use a uniform refinement scheme across the shortest diagonal to obtain the following results.
\begin{table}[H]
	\begin{center}
		\begin{tabular}{|c|c|c|c|} 
			\hline
			No. Nodes & No. Elements & $\min_{x \in \Omega_h}\Ghz(x)$ & $\max_{x \in \Omega_h}\Ghz(x)$\\
			\hline 
			87 & 384 & 0 & 25\\ 
			\hline 
			581 & 3072 & -0.000515941 & 58.7273\\ 
			\hline 
			4329 & 24576 & -0.00000412 & 130.022\\ 
			\hline 
			33617 & 196608 & -0.00275881 & 273.436\\ 
			\hline 
			265377 & 1.57286e+06 & -0.00138705 & 566.92\\ 
			\hline 
			2.10976e+06 & 1.25829e+07 & -0.00102935 & 1139.63\\ 
			\hline 
		\end{tabular}
	\end{center}
	\caption{Discrete Green's function values using local and uniform refinement across shortest diagonal.}
	\label{table: mesh2 uniformrefinel}
\end{table}

As we can see, in each case $|\min_{x \in \Omega_h}\Ghz(x)|$ begins to eventually decrease which may give rise to potential mesh restrictions that will guarantee positivity. One could, for example, require that the stiffness matrix $A$ be a $M$-matrix. This would automatically imply positivity of the discrete Green's function. However, in three dimensions, this is quite restrictive to the triangulation. We are interested in being able to describe weaker mesh restrictions that guarantee positivity.

\section{Conclusions}
In this paper, we have established some sharp pointwise bounds on the discrete Green's function. In particular we showed that in three dimensions the discrete Green's function cannot be uniformly bounded in $h$ at the singularity. Furthermore, we show that away from the singularity the discrete Green's function is positive and decays exponentially to the boundary. We also prove analogous results for the discrete Green's function showing similarities to the continuous Green's function. Numerically, we show that on an a convex domain with sufficiently smooth boundary on an unstructured mesh the discrete Green's function has persistent negative values. The results are summarized in the following table.

	\begin{table}[h]
\begin{center}
			\begin{tabular}{|M{4cm}|M{4cm}|}
			\hline
		\vspace{1pt} Let $x\in\Omh$                           & \vspace{1pt}Results for $\Ghz(x)$     \\[5pt] \hline
		\vspace{1pt}	$x=z$                                    & \vspace{1pt}$\Ghz(z)\geq Ch^{-1}$     \\[5pt] \hline
		\vspace{1pt}	$x\in\Omh\backslash B_{Kh^{\frac{1}{2}}}(z)$ & \vspace{1pt}$\Ghz(x)\geq C|x-z|^{-1}$ \\[5pt] \hline
			\vspace{1pt}	$x\in B_{Kh^{\frac{1}{2}}}(z)$               & \vspace{1pt}$\Ghz(x)<0$*              \\[5pt] \hline
				\end{tabular}
\end{center}
			\caption{Summary of results for the discrete Green's function. *\textit{Numerical Result}.}
			\end{table}
		
Additionally we establish $L^p$ estimates for the gradient of both the discrete and regularized Green's function for $2\leq p\leq\infty$. Furthermore, we prove a $L^1$ estimate for $\nabla\reggreenz$. At this time we suspect a similar result can be shown for $\dgreenz$. Additionally, establishing a lower bound for $\nabla\dgreenz$ close to the singularity is an immediate priority. We believe these results could be used in establishing best approximation results in optimal control problems.

Going forward our goal is to theoretically establish the existence of a persistent negative value for the discrete Green's function in three dimensions on unstructured meshes. We are also interested in providing a constructive mesh for which negativity occurs as well as detailing weaker mesh restrictions that give positivity. Finally, another direction for investigation may include the extension of the Harnack inequality to the inhomogeneous case and to parabolic or more general elliptic equations on non-smooth domains.

%%%%%%%%%%%%%%%%%%%%%%%%%%%%%%%%%%%%%%%%%%%%%%%%%%%%%%%%%%%%%%%%%%%%%%%%%%%%%%%%%%%%%%%%%%%%%%%%%%%%%%%%%%%%%%%%%%%%%%%%%%%%%%%%%%%%%%%%%%%%%%%%%%%%%%%%%%%%%%%%%%%%%%%%%%%%%%%%%%%%%%%%%%%%%%%%%%%%%%%%%%%%%%%%%%%%%%%%%%%%%%%%%%%%%%%%%%%%%%%%%%%%%%%%%%%%%%%%%%%%%%%%%%%%%%%%%%%%%%%%%%%%%%%%%%%%%%%%%%%%%%%%%%%%%%%%%%%%%%%%%%%%%%%%%%%%%
\section*{Acknowledgments}
The author would like to thank Prof. Dmitriy Leykekhman for the many valuable discussions and support throughout the preparation of this paper.
 
\appendix
\section{Miscellaneous Proofs}
In the paper \textit{Variational Bounds on the Entries of a Inverse Matrix} \cite{MR1186730} the authors conclude the following result, which was critical in the proof of Theorem \ref{lemma: dgreen lowerbound} so we felt as though it were important to include a proof of the result.
\begin{lemma}\label{lemma: inverse matrix bounds}
	Assume $\veccv\in\RN^n\backslash\{\vec{0}\}$ and the matrix $A$ is symmetric. Then  the following inequalities are equivalent. 
	\begin{equation}\label{eq: inverse matrix bounds}
		\beta\leq\dfrac{\veccvT A\veccv}{ \veccvT \veccv}\leq\alpha\;\;\iff\;\;\frac{1}{\alpha}\leq\dfrac{\veccvT A^{-1}\veccv}{ \veccvT \veccv}\leq\frac{1}{\beta}.
	\end{equation}
\end{lemma}
\begin{proof}
	We shall start with the left hand result and show it is equivalent to the right hand side. The reverse implication follows immediately. That is, assume
	$$\beta\leq\dfrac{\veccvT A\veccv}{ \veccvT \veccv}\leq\alpha,$$
	holds for all $\veccv\in\RN^n\backslash\{\vec{0}\}$. Then the result also holds for $A^{-\frac{1}{2}}\veccv$. Therefore by the symmetry of $A$ we have,
	\begin{align*}
		\beta\leq\dfrac{(A^{-\frac{1}{2}}\veccv)^\top AA^{-\frac{1}{2}}\veccv}{ (A^{-\frac{1}{2}}\veccv)^\top A^{-\frac{1}{2}}\veccv }\leq\alpha&\implies\beta\leq\dfrac{\veccvT A^{-\frac{1}{2}} A^{\frac{1}{2}}\veccv}{(A^{-\frac{1}{2}}\veccv)^\top A^{-\frac{1}{2}}\veccv }\leq\alpha\\
		&\implies\beta\leq\dfrac{\veccvT A^{-\frac{1}{2}} A^{\frac{1}{2}}v}{\veccvT A^{-\frac{1}{2}} A^{-\frac{1}{2}}\veccv}\leq\alpha\\
		&\implies\beta\leq\dfrac{\veccvT \veccv}{\veccvT A^{-1}\veccv }\leq\alpha\\
		&\implies\frac{1}{\alpha}\leq\dfrac{\veccvT A^{-1}\veccv}{ \veccvT \veccv}\leq\frac{1}{\beta}.
	\end{align*}
\end{proof}

%%%%%%%%%%%%%%%%%%%%%%%%%%%%%%%%%%%%%%%%%%%%%%%%%%%%%%%%%%%%%%
\bibliographystyle{siam}
\bibliography{myrefdiscgreenpos}

\end{document}